\documentclass[11pt]{amsart}
\usepackage[cp1251]{inputenc}

\newtheorem{theorem}{Theorem}[section]
\newtheorem{lemma}[theorem]{Lemma}
\newtheorem{proposition}[theorem]{Proposition}
\newtheorem{corollary}[theorem]{Corollary}

\newtheorem{conjecture}[theorem]{Conjecture}

\theoremstyle{definition}
\newtheorem{definition}[theorem]{Definition}

\newcommand{\Ad}{\operatorname{Ad}}

\newcommand{\Aut}{\operatorname{Aut}}
\newcommand{\TAut}{\operatorname{TAut}}

\newcommand{\GL}{\mbox{GL}}

\def\Ind{\operatorname{Ind}}
\def\SL{\operatorname{SL}}
\def\End{\operatorname{End}}

\def\Id{\operatorname{Id}}
\def\Ch{\operatorname{Char}}

\theoremstyle{remark}
\usepackage{amscd,amssymb}
\begin{document}
\vskip-5cm

\baselineskip 19 pt

\title[Automorphisms of the Tame Polynomial Automorphism Group] {On Automorphisms of the Tame Polynomial Automorphism Group in Positive Characteristic}
\voffset-2cm

\author[A. Belov-Kanel, A. Elishev, and J.-T. Yu ]
{Alexei Belov-Kanel, Andrey Elishev, and Jie-Tai Yu}
\address{Department of Mathematics, Bar-Ilan University,
Ramat-Gan, 52900, Israel} \email{beloval@cs.biu.ac.il,\ kanelster@gmail.com}
\address{Department of Discrete Mathematics, Moscow Institute of Physics and Technology, Dolgoprudny, Moscow Region, 141700, Russia}
\email{elishev@phystech.edu}
\address{College of Mathematics and Statistics, Shenzhen University, Shenzhen, 518061, China} \email{jietaiyu@szu.edu.cn}

\thanks{Supported by the Russian Science Foundation grant No. 17-11-01377.}



\begin{abstract}
In this note, which is a continuation of the paper \cite{KBYu}, we prove that over algebraically closed field $K$ of positive
characteristic $\neq 2$ every automorphism of the group of origin-preserving automorphisms of the polynomial algebra
$K[x_1,\ldots, x_n]$ ($n>3$) which fixes every diagonal matrix preserves, up to composition with a linear inner
automorphism, every tame automorphism.
\end{abstract}

\maketitle

\qquad \qquad Dedicated to Prof. B.I. Plotkin on the occasion of his 95th birthday.

\sloppy

\section{Introduction}
This paper serves as commentary and supplement to the previous study \cite{KBYu} concerning $\Ind$-schemes of polynomial
algebra automorphisms defined over algebraically closed field.

The Jacobian conjecture of Keller \cite{Keller} is a notoriously hard open problem in the theory of polynomial mappings, which
over the course of its existence has been a source of interesting research directions in algebra and algebraic geometry (see
\cite{BBRY, BCW, vdE, Moh1} and references therein). Relatively recently a connection \cite{K-BK1, K-BK2, Tsu2, Tsu1}
between the Jacobian conjecture and the conjecture of Dixmier \cite{Dix} on the endomorphisms of the Weyl algebra has been
discovered.

The study of $\Ind$-schemes of polynomial endomorphisms and automorphisms as an approach to the Jacobian conjecture goes
back to Shafarevich \cite{Shafarevich}, who introduced the concept of $\Ind$-scheme and studied it in detail. An important
result linking together the geometry, algebra, and combinatorics of this area is the theorem of Anick \cite{An} on approximation,
in the formal power series topology, of polynomial endomorphisms with constant non-zero Jacobian by so-called tame
automorphisms.

As it turns out, tame automorphisms play a prominent role in the connection between the Jacobian and the Dixmier conjectures
\cite{K-BK1}. In particular, in \cite{K-BK1} an isomorphism between subgroups of tame automorphisms of the Weyl and
Poisson algebras has been discovered; this has led the authors of \cite{K-BK1} to conjecture the isomorphism between the entire
automorphism groups. Also, a symplectic (i.e. for Poisson polynomial algebra) version of Anick's theorem holds, as was
established in \cite{KGE}, which allows, with certain appropriate modifications, to attack the aforementioned isomorphism
conjecture with approximation techniques.

In light of Anick's approximation theorem, Shafarevich's approach, as well as the stable equivalence between the Jacobian and the
Dixmier conjectures, it is sensible to devote resources to the study of automorphisms and $\Ind$-automorphisms (i.e.
automorphisms which are algebraic and respect the $\Ind$-structure) of the groups ($\Ind$-schemes) of polynomial
automorphisms of various algebras (such as: polynomial algebra $K[x_1,\ldots, x_n]$, free algebra $K\langle z_1,\ldots,
z_n\rangle$, Poisson algebra, Weyl algebra etc.). The theory of such objects also has a different origin in the universal algebra;
this was pioneered by B. I. Plotkin \cite{BIP2, BIP1}. Furthermore, from a purely algebro-geometric viewpoint this theory
belongs to the realm of affine algebraic geometry, particularly to the geometry of automorphism groups of varieties. An extensive
treatise on the subject is due to Furter and Kraft \cite{FurKr}.

In \cite{KBYu}, the study of $\Ind$-schemes was merged with topological (approximation) results of Anick in order to shed
light on certain aspects of the structure of $\Aut\Aut$. In particular, it has been demonstrated that over an algebraically closed
field $K$ of characteristic zero every $\Ind$-automorphism of $\Aut K[x_1,\ldots, x_n]$ ($n\geq 3$) is inner. A large part of the
argument is of combinatorial nature, and it was natural to try and extend it to the positive-characteristic case. The cornerstone of
the argument -- or rather, that for the case $n>3$ (the number of algebra generators) -- is the Proposition 4.3 of \cite{KBYu},
originally due to E. Rips, which in turn relies on a certain statement, referred to in \cite{KBYu} (Lemma 4.12) as the lemma of
Rips, on the generators of the subgroup of tame automorphisms.

The argument works well in characteristic zero. However in the case of positive characteristic certain complications related to the
structure of spaces of homogeneous polynomials present themselves. This paper aims to address the issue by circumventing the
problem through partial adaptation of the combinatorial argument from the case $n=3$ to $n>3$. The lemma of Rips in the
positive characteristic thus becomes an interesting side conjecture.

The main body of the paper (Section 2) consists of the subsection containing necessary preliminaries, the subsection devoted to
the study of implications of maximal torus stabilization for the action on linear automorphisms, and, lastly, the subsection
dedicated to the proof of the main result (Theorem \ref{thmripsmain}).

\subsection{Acknowledgements}


We thank Andrei M. Raigorodskii for fruitful and stimulating discussions and numerous helpful suggestions on improvement of
this paper.

The work is supported by the RSF Grant No. 17-11-01377.

\section{Automorphisms of $\Aut$ and $\TAut$}

\subsection{Definitions and preliminaries}

Throughout, the base field $K$ is algebraically closed and has positive characteristic $p\neq 2$. The number of generators of the
polynomial algebra $K[x_1, \ldots, x_n]$ is $n\geq 4$. The generators $x_i$ are fixed and carry degree one.

\begin{definition}
We denote by $\Aut K[x_1, \ldots, x_n]$ the group of $K$-algebra automorphisms  of $K[x_1, \ldots, x_n]$. The group
operation $\circ$ is composition of automorphisms. Every element $\varphi\in \Aut K[x_1, \ldots, x_n]$ is specified by an
$n$-tuple of polynomials
$$
(F_1(x_1,\ldots, x_n),\ldots, F_n(x_1,\ldots, x_n)),\;\;F_i(x_1,\ldots, x_n) \equiv \varphi(x_i)
$$
-- the images of algebra generators $x_i$ under $\varphi$. An arbitrary $n$-tuple $(F_1,\ldots, F_n)$ of polynomials is the same
as an algebra endomorphism $\varphi \in\End K[x_1, \ldots, x_n]$ (the correspondence is one-to-one); in order to specify an
automorphism -- i.e. an invertible endomorphism -- the coefficients of $(F_1,\ldots, F_n)$ must satisfy certain identities.
\end{definition}

\begin{definition}
The subgroup $\Aut_0 K[x_1, \ldots, x_n]$ of $\Aut K[x_1, \ldots, x_n]$ is by definition the group of origin-preserving
automorphisms; its elements are given by $n$-tuples $(\varphi(x_1), \ldots, \varphi(x_n))$, with the degree-zero term of all
$\varphi(x_i)$ being zero.
\end{definition}

We will work with $\Aut_0 K[x_1, \ldots, x_n]$ in this paper.

\begin{definition}
An automorphism $\varphi\in \Aut_0 K[x_1, \ldots, x_n]$ is linear, if all $\varphi(x_i)$ are linear, i.e. of the form
$$
\varphi(x_i) = \sum_{j=1}^n a_{ij}x_j,\;\;a_{ij}\in K.
$$
\end{definition}

Obviously, $\varphi$ is a linear automorphism iff the matrix $[a_{ij}]$ is invertible. We denote by $\GL_n(K)$ the subgroup of
linear automorphisms.

\smallskip

\begin{definition} \label{defelementaryaut}
  An automorphism $\varphi\in \Aut K[x_1, \ldots, x_n]$ is called elementary if it is of the form
  $$
(x_1,\ldots, x_i + P(x_1,\ldots, \hat{x}_i,\ldots, x_n), \ldots, x_n)
  $$
  -- i.e. if it acts as the identity map on all generators except one, $x_i$, which is mapped to $x_i + P(x_1,\ldots, \hat{x}_i,\ldots, x_n)$, where the polynomial $P$ does not
  depend on $x_i$.
\end{definition}

Note that the following elementary and useful property holds: if
$$
\varphi: x_i \to x_i + P(x_1, \ldots, \hat{x}_i,\ldots, x_n)
$$
is elementary, and
$$
P = \sum_k M_k
$$
is the decomposition of $P$ into the sum of its monomials, then $\varphi$ may be presented as a composition of elementary
automorphisms
$$
x_i\to x_i + M_k
$$
corresponding to the monomials $M_k$ (these elementaries evidently commute with one another). We will frequently utilize this
property in the proofs below.

\smallskip

The subgroup $\GL_n(K)$ of linear (origin-preserving) automorphisms is, as a group, essentially the same as the group of all
invertible over $K$ matrices. We will therefore prefer not to make a distinction between the two points of view and use the terms
"linear automorphism" and "matrix" interchangeably.

\begin{definition}
  We denote by $\TAut K[x_1, \ldots, x_n]$ the subgroup generated by $\GL_n(K)$ together with all elementary automorphisms. Elements of $\TAut K[x_1, \ldots, x_n]$ are called tame automorphisms.
\end{definition}

By restricting $P$ in Definition \ref{defelementaryaut} to have zero degree-zero term, we obtain the definition of
origin-preserving elementary automorphisms and, consequently, that of the group $\TAut_0 K[x_1, \ldots, x_n]$ of
origin-preserving tame automorphisms.

Henceforth we will generally mean by "elementary automorphisms" those which preserve the origin, in accordance with our
general convention to restrict our attention to $\Aut_0$.

\smallskip

\begin{definition}\label{defautaut}
The groups $\Aut \Aut K[x_1, \ldots, x_n]$, $\Aut \Aut_0 K[x_1, \ldots, x_n]$, \\
$\Aut \TAut K[x_1, \ldots, x_n]$ and $\Aut \TAut_0 K[x_1, \ldots, x_n]$ consist of group automorphisms of the respective
groups.
\end{definition}

Thus, an element $\Phi$ of $\Aut \Aut_0$, for example, takes an origin-preserving automorphism $\varphi$ and returns the
origin-preserving automorphism $\Phi(\varphi)$; furthermore, $\Phi$ is one-to-one, the automorphism composition is preserved,
and the identity automorphism is mapped to itself.

It is well known that $\Aut K[x_1, \ldots, x_n]$ admits the structure of an $\Ind$-scheme via filtration by total degree; this can
be made manifest by embedding $\Aut K[x_1, \ldots, x_n]$ into $\End K[x_1, \ldots, x_n]\times \End K[x_1, \ldots, x_n]$ as
$$
\varphi\mapsto (\varphi, \varphi^{-1})
$$
and imposing the polynomial identities on the coefficients of $\varphi$, $\varphi^{-1}$ stemming from the identities
$\varphi\circ \varphi^{-1} = \varphi^{-1}\circ \varphi = \Id$. Filtration by total degree of $\varphi$ and $\varphi^{-1}$ then
leads to an inductive system of schemes over $K$ whose direct limit corresponds to $\Aut K[x_1, \ldots, x_n]$.

This immediately leads to the notion of $\Ind$-morphism as a mapping which restricts to a scheme morphism on the components
of the filtration and, particularly, the definition of the subgroup $\Aut_{\Ind}\Aut$ of $\Ind$-automorphisms.



Let $K$ be algebraically closed.

\begin{definition}
  An $m$-torus $T^m$ over $K$ is the set $T^m = (K^{\times})^m$ together with abelian group structure defined by coordinate-wise multiplication.
\end{definition}

An $n$-torus $T^n$ is naturally embedded in $\GL_n(K)$: a point $(\lambda_1,\ldots, \lambda_n)$ corresponds to the linear
automorphism
$$
\sigma_{\lambda}: (x_1,\ldots, x_n)\rightarrow (\lambda_1x_1,\ldots, \lambda_nx_n).
$$
We will call this embedding $T^n\subset \Aut_0 K[x_1, \ldots, x_n]$ standard and usually refer to its image as the standard, or
maximal, torus.

More generally, one may speak of torus actions on the polynomial algebra $K[x_1, \ldots, x_n]$:
$$
\sigma: T^m\rightarrow \Aut K[x_1, \ldots, x_n].
$$

There is a well known theorem, due to A. Bia\l{}ynicki-Birula \cite{Bialynicki-Birula1,Bialynicki-Birula2}, that states that in the
case of algebraically closed field $K$, any faithful (i.e. with trivial kernel) maximal torus action $\sigma$ is conjugate to a
representation: there is an automorphism $\varphi$ (coordinate change) such that $\varphi\circ \sigma\circ\varphi^{-1}$ is linear.
In fact, the latter can always be made diagonal and, since maximal tori over algebraically closed field lie in one conjugacy class in
$\Aut K[x_1, \ldots, x_n]$, the diagonal action can be made to correspond to the standard embedding of $T^n$.

Torus actions provide a rather powerful combinatorial tool in the study of automorphisms of polynomial automorphism groups,
as will be seen later in the course of this paper. Also, as many of the problems discussed here admit analogues in the case of the
free associative algebra (cf. \cite{KBYu} for details), the tools provided in the commutative polynomial case by the
Bia\l{}ynicki-Birula had to be adapted to the noncommutative case. This has been done in \cite{TA1} in a precise analogy with
the commutative case.

\subsection{Action of $\Phi$ on elementary automorphisms -- the approach of Rips}

The main objective of this paper is the proof of the following result, originally formulated by Rips, for the case of base field of
positive characteristic not equal to $2$.
\begin{theorem}[Rips]\label{thmripsmain}
Let $n>3$. Suppose $$\Phi:\Aut_0 K[x_1, \ldots, x_n]\rightarrow\Aut_0 K[x_1, \ldots, x_n]$$ preserves point-wise the
standard torus action $T^n\subset \Aut_0 K[x_1, \ldots, x_n]$. Then there is a linear automorphism $S\in \GL_n(K)$ such that
the composition of $\Phi$ with the adjoint action of $S$ preserves all elementary automorphisms.
\end{theorem}

The immediate corollary of this theorem is the point-wise preservation of all tame automorphisms under the composition of
$\Phi$ with an inner automorphism of $\TAut_0$.
\begin{corollary}
Let $\Phi$ satisfy the conditions of Theorem \ref{thmripsmain}. Then there is a linear automorphism $S\in \GL_n(K)$ such that
the composition of $\Phi$ with the adjoint action of $S$ preserves all tame automorphisms.
\end{corollary}


In the case of base field of characteristic zero, Theorem \ref{thmripsmain} has been established in \cite{KBYu} (for the case of
$n>3$ generators discussed here as well as for the case $n=3$). The proof relies on the following lemma, also originally due to
Rips.

\begin{lemma}[Rips]  \label{LmRips}
Let $\Ch K = 0$. Then linear automorphisms, together with the automorphism
$$
\psi: x_1\to x_1 + x_2x_3,\;\;x_2\to x_2,\;\;\ldots,\;\; x_n\to x_n
$$
generate the entire (origin-preserving) tame automorphism group of $K[x_1,\ldots,x_n]$, when $n>3$.
\end{lemma}

This result is stated in \cite{KBYu} as Lemma 4.12 -- for the case $\Ch K = 0$ as well as $\Ch K = p\neq 2$. The proof given in
the initial paper, however, works only in the characteristic zero case, while the positive characteristic case contains a subtlety
contributing to the breakdown of the final stage of the combinatorial argument. Therefore, the positive characteristic version of
the lemma of Rips needs to be reformulated as a conjecture:

\begin{conjecture}[Rips]\label{conjrips}
Let $K$ be an infinite base field with $\Ch K = p\neq 2$. Then linear automorphisms, together with the automorphism
$$
\psi: x_1\to x_1 + x_2x_3,\;\;x_2\to x_2,\;\;\ldots,\;\; x_n\to x_n
$$
generate the entire tame automorphism group of $K[x_1,\ldots,x_n]$, when $n>3$.
\end{conjecture}

This statement, to which we will refer as the \textbf{Rips conjecture} in this paper, is rather interesting. At present, we do not
possess compelling evidence either in favor or against the statement of this conjecture, which makes the investigation of this
problem into an intriguing and worthwhile endeavor.

In order to rectify the proof of the positive characteristic case of Theorem \ref{thmripsmain} in a way that circumvents the yet to
be discovered nature of the generators of the tame automorphism subgroup, we proceed according to the following plan:

1. Observe that the preservation of the maximal torus by $\Phi$ implies, up to a linear inner automorphism, preservation of
$\GL_n(K)$ as well as
$$
\psi: x_1\to x_1 + x_2x_3,\;\;x_2\to x_2,\;\;\ldots,\;\; x_n\to x_n.
$$

2. Establish that on elementary automorphisms of the form
$$
x_j\to x_j +M,
$$
where $M$ is a monomial of degree $\geq 2$, the action of $\Phi$ yields
$$
x_j\to x_j + aM
$$
where $a\in K^{\times}$ depends only on the $K$-submodule generated by $M$.

3. Demonstrate that $a=1$ always holds. To that end, distinguish the following cases:

   -- Case A (the "\textbf{good}" case). The monomial
   $$
M = x_1^{k_1}\ldots x_n^{k_n}
   $$
   contains an $x_i$ with $k_i+1$ a multiple of $p$.

   -- Case B (the "\textbf{bad}" case). All powers $k_i$ contributing to $M$ are such that $k_i+1$ is divisible by $p$.

In the last step of this plan, the first case is processed in two steps by induction on the number of variables in $M$; while the
second case essentially relies on the first one.

\subsection{Action of $\Phi$ on $\GL_n(K)$ and $\psi$}

In this subsection, we perform the first step of the outlined plan of the proof; starting with the assumption that
$$
\Phi: \Aut_0 K[x_1,\ldots,x_n]\rightarrow \Aut_0 K[x_1,\ldots,x_n]
$$
preserves points of the maximal torus $T^n$, we show that the composition of $\Phi$ with a linear inner automorphism preserves
point-wise the subgroup $\GL_n(K)$ of linear automorphisms together with the quadratic elementary automorphism
$$
\psi: x_1\to x_1+x_2x_3.
$$
This result was in fact established in the initial paper \cite{KBYu}; we reproduce the proof here for the sake of completeness. In
actuality, we will be primarily concerned with verifying the statement regarding the quadratic automorphism $\psi$, as the
theorem stating that the automorphism
$$
\Phi: \GL_n(K)\rightarrow \GL_n(K)
$$
preserving every diagonal matrix is inner is a fairly classical result. Thus the focus of this subsection lies in showing that an
element $\Phi\in \Aut \Aut_0 K[x_1,\ldots, x_n]$ which preserves every point of $\GL_n(K)$ can be composed with a linear
inner automorphism so that the resulting mapping $\tilde{\Phi}$ preserves the quadratic automorphism $\psi$ as well as every
linear automorphism. We will also comment on how to obtain the stabilization of $\GL_n(K)$ from stabilization of the maximal
torus.

Firstly, the following lemma is established by a trivial, direct computation.

\begin{lemma}  \label{LmTorComp}
Consider the linear automorphisms: $\alpha: x_i\to \alpha_i x_i$, $\beta: x_i\to \beta_i x_i$. Let $$\varphi: x_i\to \sum_{i,J}
a_{iJ}x_J,\; i=1,\dots,n$$ be a general polynomial automorphism, written explicitly, where $J=(j_1,\dots,j_n)$ is the
multi-index: $x_J=x_{j_1}\cdots x_{j_n}$. Then

$$\alpha\circ\varphi\circ\beta: x_i\to \sum_{i,J}\beta_i \alpha_J
a_{iJ}x_J.$$

In particular,
$$\alpha\circ\varphi\circ\alpha^{-1}: x_i\to \sum_{i,J} \alpha_i^{-1}\alpha_{J}
a_{iJ}x_J.$$
\end{lemma}

Applying Lemma \ref{LmTorComp} and comparing the coefficients we get the following

\begin{lemma}  \label{LmLindiag}
Consider the diagonal $T^1$ action: $x_i\mapsto \lambda x_i$. Then the set of automorphisms commuting with this action is
exactly the set of linear automorphisms.
\end{lemma}


The previous lemma allows one to prove the following proposition.
\begin{proposition}
Let $\Phi$ preserve point-wise the maximal torus. Then for all $\varphi\in \GL_n(K)$, $\Phi(\varphi) \in \GL_n(K)$, i.e. $\Phi$
maps linear automorphisms to linear automorphisms.
\end{proposition}
\begin{proof}
  If $\varphi$ is linear, then for any $\lambda \in K^{\times}$, we have
  $$
\varphi = h_{\lambda}\circ\varphi\circ h_{\lambda}^{-1}
 $$
  where $h$ is the homothety with ratio $\lambda$ corresponding to the diagonal action of $T^1\subset T^n$, i.e. a linear automorphism which maps every generator $x_i$ to $\lambda x_i$. Then
 $$
\Phi(\varphi) = \Phi( h_{\lambda}\circ\varphi\circ h_{\lambda}^{-1}) =  h_{\lambda}\circ\Phi(\varphi)\circ h_{\lambda}^{-1},
 $$
 where in the last equality we used that $\Phi$ acts as the identity map on the maximal torus (to which the homothety $h_{\lambda}$ obviously belongs).
 Therefore, $\Phi(\varphi)$ commutes with the diagonal action of $T^1$ and by Lemma \ref{LmLindiag} is linear.
\end{proof}



We are now in the situation where automorphism
$$
\Phi: \Aut_0 K[x_1,\ldots, x_n]\rightarrow  \Aut_0 K[x_1,\ldots, x_n]
$$
restricts to an automorphism
$$
\Phi: \GL_n(K)\rightarrow \GL_n(K)
$$
of the general linear group. The structure of $\Aut \GL_n(K)$ (for $n\geq 3$ and $K$ a field) is a well understood classical
subject, to which belongs the following theorem.

\begin{theorem}\label{thmlinearstab}
  Suppose $K$ is an infinite field (of arbitrary characteristic), $n\geq 3$ and the automorphism
  $$
\Phi: \GL_n(K)\rightarrow \GL_n(K)
$$
preserves every diagonal matrix. Then $\Phi$ is inner.
\end{theorem}

The proof of this theorem is contained in Chapter 3 of the classical textbook \cite{HaOM} by Hahn and O'Meara, although it may
take the interested reader some effort to track the details of proofs and references supplied in the text.

There exists a more direct approach to Theorem \ref{thmlinearstab} in the case when $\Ch K\neq 2$ (which is of primary interest
in this paper), which utilizes a theorem on the structure of automorphisms of $\GL_n(K)$ due to McDonald \cite{McD} and
Waterhouse \cite{Wat}. The latter theorem is as follows.

\begin{theorem}[McDonald \cite{McD}, Waterhouse \cite{Wat}]\label{thmmcdwat}
Let $K$ be a field of characteristic not equal to $2$ and let $n\geq 3$. Then every automorphism $\Phi$ of $\GL_n(K)$ has the
following standard form:
$$
\Phi = \Omega \circ \Ad_S^{\theta}\circ P_{\chi};
$$
-- here $P_{\chi}$ is the so-called radial automorphism, given by
$$
P_{\chi}(A) = \chi(A)A
$$
for some group homomorphism
$$
\chi: \GL_n(K)\rightarrow K^{\times};
$$

-- the automorphism $\Ad_S^{\theta}$ is a $\theta$-inner automorphism, i.e. an action by a base field automorphism $\theta$
followed by conjugation by a matrix $S$;

-- finally, $\Omega$ is either the identity map or the transpose-inverse automorphism
$$
A\to (A^{-1})^t.
$$
\end{theorem}

Automorphisms of the form
$$
\Omega \circ \Ad_S^{\theta}\circ P_{\chi}
$$
are referred to as standard \cite{Wat}. To be more precise, Theorem \ref{thmmcdwat} is a special case of the more general result
of McDonald and Waterhouse that applies to matrices over commutative rings $R$ in which the number $2$ is invertible; the
added complexity corresponds to such rings possibly having non-trivial idempotents (being disconnected in the sense of the
Zariski topology), which in turn allows for $\Omega$ to act as transpose-inverse on one of the components. Note also that every
two of the three distinct automorphism types comprising $\Phi$ can be interchanged in a semidirect product fashion, so that
composition of automorphisms of standard form is again an automorphism of standard form.

The McDonald -- Waterhouse theorem allows a straightforward proof of Theorem \ref{thmlinearstab}, which the reader is invited
to conduct. In particular, one can demonstrate, by considering action of $\Phi$ on diagonal matrices which have all but one
non-zero entry equal $1$ that:

-- the case of transpose-inverse $\Omega$ is ruled out when $K$ is infinite;

-- the base field automorphism $\theta$ must be the identity map;

-- and finally that $\chi$ acting on such matrices must return $1$, which implies that every diagonal matrix belongs to the kernel
of $\chi$. Then, as $\chi$ is a mapping to an abelian group $K^{\times}$ and as the commutator subgroup of $\GL_n(K)$ is
$\SL_n(K)$, one concludes that $\chi$ is the trivial homomorphism, which completes the proof.

\smallskip

We now turn to the question of stabilization of the quadratic automorphism
$$
\psi: x_1\to x_1 + x_2x_3, \; x_2\to x_2,\;\ldots,\;x_n\to x_n.
$$
Since we have established Theorem \ref{thmlinearstab}, we may now assume that $\Phi$ preserves every point of $\GL_n(K)$.
We will show that there exists a scalar matrix $H$ (a homothety) such that the composition of $\Phi$ with the adjoint action of
$H$ preserves $\psi$; the resulting automorphism $\Phi_H$ will still preserve every linear automorphism, for $H$ is in the
center of $\GL_n(K)$.

The outline of the proof is as follows. We first find a suitable torus action such that $\psi$ belongs to a set of automorphisms
commuting with that action. Then we show that $\Phi$ restricts to a mapping of the commutator of that set (of automorphisms
commuting with the torus action) to itself; furthermore, $\psi$ belongs to the commutator. Finally, as the commutator's structure
will turn out to be fairly simple, the action of $\Phi$ will necessarily be one of multiplication (of the quadratic part) by a global
constant $a\in K^{\times}$; that will provide the homothety ratio corresponding to the matrix $H$.

We proceed with the relevant statements.

\begin{lemma}\label{lemtorusquadr}
Consider the action of $(\lambda_2,\ldots, \lambda_n)\in T^{n-1}$ given by
$$
x_1\to \lambda_2\lambda_3 x_1,\; x_2\to \lambda_2 x_2, \;x_3\to \lambda_3x_3,\;\ldots, \; x_n\to \lambda_n x_n.
$$
Then the set $S$ of automorphisms commuting with this action is generated by automorphisms of the form
$$
x_1\to a_1 x_1 +\beta x_2x_3,\;x_2\to a_2x_2,\;\ldots,\; x_n\to a_nx_n,
$$
$a_i\in K^{\times},\;\beta\in K$.
\end{lemma}

Lemma \ref{lemtorusquadr} is in fact a particular case of the following statement.

\begin{lemma}\label{lemtorusgen}
Consider the action of $(\lambda_2,\ldots, \lambda_n)\in T^{n-1}$ given by
$$
x_1\to \lambda_2^{i_2}\ldots\lambda_n^{i_n} x_1,\; x_2\to \lambda_2 x_2, \;x_3\to \lambda_3x_3,\;\ldots, \; x_n\to \lambda_n x_n,
$$
where $i_2,\ldots, i_n\in \mathbb{Z}_+$ are such that $i_2+\cdots + i_n>1$. Then the set $S$ of automorphisms commuting
with this action is generated by automorphisms of the form
$$
x_1\to a_1 x_1 +\beta x_2^{i_2}\ldots x_n^{i_n},\;x_2\to a_2x_2,\;\ldots,\; x_n\to a_nx_n,
$$
$a_i\in K^{\times},\;\beta\in K$.
\end{lemma}
\begin{proof}
This lemma follows from Lemma \ref{LmTorComp}; the application of the general statement of Lemma \ref{LmTorComp} is
direct.
\end{proof}

It is easy to see that the set $S$ from the preceding lemma actually consists of automorphisms of the form
$$
x_1\to a_1 x_1 +\beta x_2^{i_2}\ldots x_n^{i_n},\;x_2\to a_2x_2,\;\ldots,\; x_n\to a_nx_n,
$$
(which in this case is equivalent to being generated by such automorphisms), and that $S$ forms a subgroup of $\Aut_0
K[x_1,\ldots, x_n]$. The significance of Lemma \ref{lemtorusgen} is now apparent: if $\Phi$ is an automorphism of $\Aut_0
K[x_1,\ldots, x_n]$ for which every point of the maximal torus is a fixed point, then
$$
\Phi_{|S}: S\rightarrow S
$$
is an automorphism of $S$; in other words, $\Phi$ maps $S$ to itself. Indeed, if $\psi\in S$ is an automorphism that commutes
with the action of $T^{n-1}\subset T^n$ as in Lemma \ref{lemtorusgen} and $\Lambda$ represents a generic point of
$T^{n-1}$, then
$$
\Phi(\psi) = \Phi(\Lambda\circ\psi\circ \Lambda^{-1}) = \Lambda\circ \Phi(\psi)\circ \Lambda^{-1}
$$
which means that $\Phi(\psi)$ commutes with $\Lambda$, and is therefore an element of $S$. Thus the following lemma is
proved.
\begin{lemma}\label{lemphirestriction}
Let $S$ be as in Lemma \ref{lemtorusgen} and $\Phi$ an automorphism for which every point of $T^n$ is fixed. Then $\Phi$
maps $S$ to itself.
\end{lemma}

Furthermore, as $\Phi$ is homomorphic, we have the following property.

\begin{lemma}\label{lemphicommutant}
Let $S$ be as in Lemma \ref{lemtorusgen} and $[S,S]$ denote the commutator subgroup of $S$. Then $\Phi$ restricts to an
automorphism of $[S,S]$.
\end{lemma}

Suppose now that $S$ is the set of automorphisms which commute with the action of $T^{n-1}$ as in Lemma
\ref{lemtorusquadr}:
$$
x_1\to \lambda_2\lambda_3 x_1,\; x_2\to \lambda_2 x_2, \;x_3\to \lambda_3x_3,\;\ldots, \; x_n\to \lambda_n x_n.
$$
Then the following holds.
\begin{lemma}\label{lemcommsubgrquadr}
 The commutator subgroup $[S,S]$ of the group $S$ consists precisely of automorphisms of the form
$$
x_1\to x_1 +\beta x_2x_3,\;x_2\to x_2,\;\ldots,\; x_n\to x_n,
$$
$\beta \in K$.
\end{lemma}
\begin{proof}
  Every element of $[S,S]$ is verified to be of the stated form by direct computation. It is also not difficult to establish that, when $K$ is infinite, every $\beta \in K$ represents an element of $[S,S]$.
\end{proof}

Now, $\Phi$ acts as an automorphism of $[S,S]$, which means that its action on an arbitrary element $\varphi$ of the form
$$
x_1\to x_1 +\beta x_2x_3,\;x_2\to x_2,\;\ldots,\; x_n\to x_n
$$
yields the element $\Phi(\varphi)$ of the form
$$
x_1\to x_1 +a\beta x_2x_3,\;x_2\to x_2,\;\ldots,\; x_n\to x_n,
$$
with $a=a_{\varphi}\in K^{\times}$ whenever $\beta\in K^{\times}$. The action of $\Phi$ thus reduces on $[S,S]$ to a
multiplication of the quadratic term by a number. Our immediate objective is to demonstrate that this number $a_{\varphi}$ is
actually a constant.
\begin{lemma}\label{lemglobalconstquadr}
In the notation above, $a_{\varphi}$ is independent of $\varphi$.
\end{lemma}
\begin{proof}
  Indeed, let $\varphi$ be given and $a = a_{\varphi}$ is the multiplier corresponding to the action of $\Phi$ on $\varphi$. Consider the automorphism $\psi$ given by
  $$
x_1\to x_1 + x_2x_3,\;x_2\to x_2,\;\ldots,\; x_n\to x_n
  $$
(i.e. it is $\varphi$ with $\beta = 1$). Then $\Phi(\psi)$ is
$$
x_1\to x_1 +a_{\psi} x_2x_3,\;x_2\to x_2,\;\ldots,\; x_n\to x_n
$$
for some $a_{\psi}\neq 0$. Now observe that $\varphi$ (for $\beta\neq 0$) and $\psi$ are conjugate to one another by an element
of the torus:
$$
\varphi = \Lambda\circ \psi\circ\Lambda^{-1}
$$
with
$$
\Lambda: x_1\to x_1,\;\;x_2\to \beta x_2,\;\;x_3\to x_3,\;\;\ldots,\;\;x_n\to x_n.
$$
Therefore, as points of the torus are stabilized by $\Phi$, we have
$$
\Phi(\varphi) = \Lambda\circ \Phi(\psi)\circ\Lambda^{-1}
$$
which immediately implies that
$$
a_{\varphi}\beta = a_{\psi}\beta,
$$
or
$$
a_{\varphi} = a_{\psi} = a
$$
as desired.
\end{proof}

The main result of this section (on the stabilization of $\psi$) now follows.
\begin{theorem}\label{thmstabquadr}
Suppose $\Phi$ is an automorphism of $\Aut_0 K[x_1,\ldots, x_n]$ which preserves point-wise the subgroup $\GL_n(K)$ of
linear automorphisms. Then there is a homothety
$$
H: x_1\to hx_1,\;\;\ldots,\;\; x_n\to hx_n,\;\;h\in K^{\times}
$$
such that the automorphism $\Phi_H$ obtained by composing $\Phi$ with the inner automorphism given by conjugation by $H$
preserves point-wise $\GL_n(K)$ as well as the quadratic automorphism
$$
\psi: x_1\to x_1+x_2x_3,\;\;x_2\to x_2,\;\;\ldots,\;\; x_n\to x_n.
$$
\end{theorem}
\begin{proof}
 If $\Phi$ is given, then the homothety $H$ is specified by the homothety ratio $h=a^{-1}$, where $a$ is the constant of Lemma \ref{lemglobalconstquadr}.
 Indeed, as the center of $\GL_n(K)$ is given by homothety automorphisms (scalar matrices), the composition of $\Phi$ with the inner automorphism corresponding to $H$ does not impact point-wise
 preservation of $\GL_n(K)$. Furthermore, if $\psi$ is as before,
 $$
\psi: x_1\to x_1+x_2x_3,\;\;x_2\to x_2,\;\;\ldots,\;\; x_n\to x_n
 $$
 and $\Phi(\psi)$ is
 $$
x_1\to x_1+ax_2x_3,\;\;x_2\to x_2,\;\;\ldots,\;\; x_n\to x_n
 $$
for $a\in K^{\times}$, then (mapping $x_1$ step by step):
$$
H\circ \Phi(\psi)\circ H^{-1}: x_1\to ax_1\to a(x_1 + ax_2x_3) \to a(a^{-1}x_1 + a^{-1}x_2x_3) = \psi(x_1),
$$
the action on the rest of the generators being obvious. This means exactly that $\psi$ is stabilized by the composite
automorphism $\Phi_H$ which maps $\varphi$ to the automorphism conjugated via $H$ to $\Phi(\varphi)$.
\end{proof}

Finally, combining this result with the previously established facts about automorphisms which preserve the maximal torus, we
arrive at the main result of this subsection.
\begin{theorem}\label{thmstabquadrmain}
Suppose $\Phi$ is an automorphism of $\Aut_0 K[x_1,\ldots, x_n]$ which preserves point-wise the maximal torus $T^n$. Then
there exists a linear automorphism $S$ such that the automorphism $\Phi_S$ given by the composition of $\Phi$ with the inner
automorphism specified by $S$ preserves point-wise $\GL_n(K)$ and the quadratic automorphism $\psi$.
\end{theorem}

\subsection{Action on higher degree elementary automorphisms}

Since we have established Theorem \ref{thmstabquadrmain}, the proof of the main result of this paper -- Theorem
\ref{thmripsmain} -- will be completed once we demonstrate that every automorphism $\Phi$ which fixes every point of
$\GL_n(K)$ as well as the quadratic automorphism $\psi$ must also fix every elementary automorphism. We note (cf. Definition
\ref{defelementaryaut} and the comment immediately afterwards; also, generator renaming is a linear automorphism) that it
suffices to establish preservation of elementary automorphisms of the form
$$
x_1\to x_1 + M,\;\; x_2\to x_2,\;\;\ldots,\;\; x_n\to x_n
$$
with $M$ a monomial of total degree $\geq 2$ free of $x_1$.

First and foremost, Lemmas \ref{lemtorusgen}, \ref{lemphirestriction} and \ref{lemphicommutant} together with the immediate
analogs of Lemmas \ref{lemcommsubgrquadr} and \ref{lemglobalconstquadr} imply the following result.
\begin{proposition}\label{propphiaction}
Let $\Phi$ be an automorphism of $\Aut_0 K[x_1,\ldots, x_n]$ which fixes every point of the maximal torus. Let $\varphi$ be
an automorphism of the form
$$
x_1\to x_1 + M,\;\; x_2\to x_2,\;\;\ldots,\;\; x_n\to x_n
$$
with $M = \beta x_2^{i_2}\ldots x_n^{i_n}$, $\beta\in K^{\times}$ and $i_2+\cdots+i_n\geq 2$. Then
$$
\Phi(\varphi): x_1\to x_1 + aM,\;\;x_2\to x_2,\;\;\ldots,\;\; x_n\to x_n
$$
with $a\in K^{\times}$ a constant for every fixed multi-index $(i_2,\ldots, i_n)$ (i.e. independent of $\beta$).
\end{proposition}


Our main objective therefore reduces to demonstrating that whenever $\Phi$ stabilizes $\GL_n(K)$ together with $\psi$ the
constant $a$ from Proposition \ref{propphiaction} is always equal to $1$. The proof requires some further preparation, which we
proceed to make.

We begin by citing the following important technical result, which appears as Lemma 4.13 in \cite{KBYu}.

\begin{lemma}  \label{LmR1}
Linear transformations of $K^3$ and $$\psi: x\to x,\; y\to y,\; z\to z+xy$$ generate all mappings of the form
$$\phi_m^b: x\to x,\; y\to y,\; z\to z+bx^m,\;\; b\in K.$$
\end{lemma}
\begin{proof}
  We proceed by induction.
Suppose we have an automorphism $$\phi^b_{m-1}: x\to x,\; y\to y,\; z\to z+bx^{m-1}.$$  Conjugating by the linear
transformation ($z\to y,\; y\to z,\;x\to x$), we obtain the automorphism
$$\theta: x\to x,\; y\to y+bx^{m-1},\; z\to z.$$
Composing this on the right by $\psi$, we get the automorphism
$$\varphi: x\to x,\; y\to y+bx^{m-1},\; z\to z+yx+bx^m.$$ Note
that
$$\varphi\circ\theta^{-1}: x\to x,\; y\to y,\;
z\to z+xy+bx^m.$$
Now we see that
$$\psi^{-1}\circ\varphi\circ\theta^{-1}=\phi^b_m$$ and
the lemma is proved.
\end{proof}

In our situation, $x,y$ and $z$ can be set to be any three distinct generators among $x_1,\ldots, x_n$; as renaming the generators
is an automorphism from $\GL_n(K)$ and we assume $\Phi$ to stabilize all such automorphisms, Lemma \ref{LmR1} has the
following consequence.

\begin{lemma}\label{lemsubgroupcontents}
Define the subgroup $G_{\psi}\subset \Aut_0 K[x_1,\ldots, x_n]$ as the group generated by $\GL_n(K)$ and
$$
\psi: x_1\to x_1+x_2x_3,\;\;x_2\to x_2,\;\;\ldots,\;\; x_n\to x_n.
$$
Then $G_{\psi}$ contains:

-- automorphisms
$$
\psi_i: x_i\mapsto x_i+x_{i+1}x_{i+2},\; x_k\mapsto x_k,\; k\ne i
$$
where $i$ runs from $1$ to $n$ and the labelling of generators is cyclic;

-- every automorphism of the form
$$
\varphi_{i,j,k,\beta}: x_i\to x_i +\beta x_j^k,\;\; x_l\to x_l\;\;(l\neq i)
$$
$k\geq 1$, $\beta\in K$, $i,j = 1,\ldots, n$, $j\neq i$.
\end{lemma}

We are now ready to proceed with the proof of the \textbf{main result} (Theorem \ref{thmripsmain}) in accordance with the plan
stated at the beginning of this subsection. In order to establish that the multiplication constants $a$ by which $\Phi$ acts on
various monomials $M$, when
$$
\varphi: x_1\to x_1 + M(x_2,\ldots, x_n),
$$
are all equal to one, we distinguish two cases, according to the composition of $M$:

-- the "\textbf{good}" case: there exists $k$ which is not congruent to $-1$ modulo $p = \Ch K$ such that a generator $x_i$ is
present in $M$ as $x_i^k$;

-- the "\textbf{bad}" case: all powers $i_2,\ldots, i_n$ of $x_2,\ldots, x_n$ in $M$ are congruent to $-1$ modulo $p$.

We proceed with the analysis of \textbf{the good case} first. Suppose $M$ is of the form\\
 $x_i^kM'(x_2,\ldots, \hat{x}_i,\ldots, x_n)$
with $i\neq 1$, $k+1$ not divisible by $p=\Ch K$ and $M'$ independent of $x_i$. We may rename the generators (utilizing
stabilization of linear mappings under $\Phi$) so that $x_i\to x_2$ and $M'$ will be a monomial in $(n-2)$ variables
$x_3,\ldots, x_n$. Such $M'$ correspond to elementary automorphisms which turn out to be an important special case in our
proof.

We process the special case -- namely of those automorphisms $\varphi$ which are of the form (we drop the prime here for
brevity and denote the monomial by $M$)
$$
\varphi: x_1\to x_1 + M(x_3,\ldots, x_n),\;\; x_2\to x_2,\;\;\ldots,\;\; x_n\to x_n.
$$
Note that, as $n$ is at least four, the monomial $M$ depends on at least two variables. We need to prove the following.

\begin{theorem}\label{thmgoodinduction}
  Let
  $$
\varphi: x_1\to x_1 + M(x_3,\ldots, x_n),\;\; x_2\to x_2,\;\;\ldots,\;\; x_n\to x_n.
$$
be an elementary automorphism of the special form (i.e. $M$ is a monomial in $x_3,\ldots, x_n$) and $\Phi$ as before. Then
$$
\Phi(\varphi) = \varphi.
$$
\end{theorem}
\begin{proof}
  We proceed by induction on the number of variables in $M$ in the following form: every automorphism which acts nontrivially on a single generator by adding to it a monomial in at most $(n-2)$ other generators
  is stable under $\Phi$.

  The base case of one variable is a special case of Lemma \ref{lemsubgroupcontents}. To arrive at the proper configuration for the induction step, consider the automorphism $\varphi$ written down in the statement of
  the theorem, set
  $$
M = x_3^k M'(x_4,\ldots, x_n)
  $$
and introduce the following automorphisms:
$$
\alpha: x_2\to x_2 + x_3^k, \;\; x_i\to x_i\;\;\ (i\neq 2),
$$
$$
\theta: x_3\to x_3 + M',\;\; x_i\to x_i \;\; (i\neq 3),
$$
and
$$
\psi: x_1\to x_1 + x_2x_3,\;\; x_i\to x_i \;\; (i\neq 1)
$$
(the last one is as before). Consider the composite automorphism $\alpha^{-1}\circ\theta^{-1}\circ \psi\circ\theta\circ\alpha$ and
its action on $x_1$, $x_2$ and $x_3$ (it evidently acts as the identity map on all the other generators). Step by step, we write:
$$
\alpha^{-1}\theta^{-1} \psi\theta\alpha: x_1\to x_1 \to x_1 \to x_1+x_2x_3\to x_1 + x_2(x_3 - M')\to x_1 + (x_2-x_3^k)(x_3-M');
$$
$$
x_2\to x_2+x_3^k\to x_2+(x_3+M')^k\to  x_2+(x_3+M')^k\to x_2 + x_3^k\to x_2;
$$
$$
x_3\to x_3 \to x_3 + M'\to x_3 + M'\to x_3 \to x_3.
$$
Thus, the action of $\alpha^{-1}\theta^{-1} \psi\theta\alpha$ is
$$
x_1\to x_1 + (x_2-x_3^k)(x_3-M'),\;\; x_i\to x_i \;\; (i\neq 1),
$$
or
$$
x_1 \to x_1 + x_2x_3 + M - x_3^{k+1}- x_2M'.
$$
The last line shows that the composite automorphism $\alpha^{-1}\theta^{-1} \psi\theta\alpha$ is actually composed of (cf.
comment after Definition \ref{defelementaryaut}) our initial automorphism $\varphi$ (corresponding to $M$), the
automorphism $\psi$, the automorphism $\varphi_{1,3,k+1,-1}$ (in the notation of Lemma \ref{lemsubgroupcontents})
corresponding to $(-x_3^{k+1})$ which belongs to the group $G_{\psi}$, and, lastly, the automorphism
$$
\varphi_1: x_1\to x_1 - x_2M'.
$$
The order in which the enumerated mappings add up to $\varphi$ is irrelevant, as all the elementary automorphisms in question
act on a single generator $x_1$ and therefore commute with one another in accordance with the comment following Definition
\ref{defelementaryaut}.

Now, observe that the automorphism $\varphi_1$ can be represented as the commutator
$$
\psi^{-1}\circ\theta^{-1}\circ
\psi\circ\theta,
$$
as direct computation of the action of the commutator on $x_1$, $x_2$ and $x_3$ shows.

We now apply $\Phi$ to the equality
$$
\varphi_1 = \psi^{-1}\circ\theta^{-1}\circ \psi\circ\theta
$$
to get
$$
\Phi(\varphi_1) = \psi^{-1}\circ \Phi(\theta)^{-1}\circ\psi\circ\Phi(\theta)
$$
as $\psi\in G_{\psi}$.

The action of $\Phi$ on the composite automorphism $\alpha^{-1}\theta^{-1} \psi\theta\alpha$ is, on the one hand, given by
$$
\alpha^{-1}\Phi(\theta)^{-1} \psi\Phi(\theta)\alpha
$$
(by Lemma \ref{lemsubgroupcontents}), while on the other hand it equals
$$
\Phi(\varphi)\psi\varphi_{1,3,k+1,-1}\Phi(\varphi_1) = \Phi(\varphi)\psi\varphi_{1,3,k+1,-1}\psi^{-1}\circ \Phi(\theta)^{-1}\circ\psi\circ\Phi(\theta).
$$
The two expressions are equal to each other, and this sets the stage for the implementation of the induction step. Namely, by
construction $\theta$ is an automorphism with monomial $M'$ on $(n-3)$ generators; the induction step now tells us that
$\Phi(\theta) = \theta$, therefore $\Phi(\varphi_1) = \varphi_1$ and, finally, from
$$
\Phi(\varphi)\psi\varphi_{1,3,k+1,-1}\varphi_1 =\alpha^{-1}\theta^{-1} \psi\theta\alpha  = \varphi\psi\varphi_{1,3,k+1,-1}\varphi_1
$$
it follows that $\Phi(\varphi) = \varphi$ as desired.
\end{proof}

The general "good" case reduces to the special case of Theorem \ref{thmgoodinduction} by means of the following argument.
Suppose
$$
\varphi: x_1\to x_1 + M(x_2,\ldots, x_n),\;\;x_2\to x_2,\;\;\ldots,\;\;x_n\to x_n
$$
is a good elementary automorphism and $M = x_2^kM'(x_3,\ldots, x_n)$ with $(k+1)$ not a multiple of $\Ch K = p$ (as always,
no generality is lost as the generators can be renamed). Then
$$
\varphi_1: x_2\to x_2 + M', \;\;x_i\to x_i\;\;(i\neq 2)
$$
is a good automorphism of the special form: by Theorem \ref{thmgoodinduction} $\Phi(\varphi_1) = \varphi_1$.

Let
$$
\alpha: x_1\to x_1 + x_2^{k+1}
$$
be an elementary automorphism. Consider the automorphism
$$
\delta = \varphi_1\circ\alpha\circ\varphi_1^{-1}.
$$
Evaluating its action on the involved generators $x_1$ and $x_2$ yields
$$
x_1\to x_1\to x_1+ x_2^{k+1}\to x_1 + (x_2+ M')^{k+1};
$$
$$
x_2\to x_2 - M' \to x_2 - M'\to x_2 +M' - M' = x_2.
$$
Therefore, $\delta$ acts nontrivially on $x_1$ only and is thus given by a product of mutually commuting elementary
automorphisms corresponding to monomials in the binomial
$$
(x_2 + M')^{k+1}.
$$
Among those is the term $x_2^kM' = M$ corresponding to the initial automorphism $\varphi$, and the term is present with
coefficient $(k+1)$ which is not zero by assumption of the good case. Therefore
$$
\delta = \varphi\circ\tilde{\varphi}
$$
where $\tilde{\varphi}$ is the product of the rest of the elementaries.

Now, by Lemma \ref{lemsubgroupcontents} and Theorem \ref{thmgoodinduction},
$$
\Phi(\delta) = \delta,
$$
which means that
$$
\Phi(\varphi)\circ\Phi(\tilde{\varphi}) = \varphi\circ\tilde{\varphi}.
$$
As $\delta$ changes $x_1$ by adding various monomials $N$, the action of $\Phi$ reduces to multiplication of every such
monomial by the corresponding constant $a = a_N$. The previous equation, on the other hand, is just an equality of two
polynomials, which means that the coefficient of every monomial on the left-hand side must be equal to the coefficient of the
corresponding monomial on the right-hand side. The contribution of $\Phi(\varphi)$ to the polynomial of the left-hand side is
$aM$ (where $a$ is as in Proposition \ref{propphiaction}), while in the right-hand side polynomial the monomial $M$ comes
from $\varphi$ with coefficient one. Therefore, $a=1$ and the good case is thus processed.

\textbf{The "bad" case} -- i.e. when
$$
\varphi: x_1\to x_1 + M(x_2,\ldots, x_n),
$$
is such that every $i_l$ in $M = b x_2^{i_2}\ldots x_n^{i_n}$ is congruent to $-1$ modulo $p$ -- can in actuality be reduced to
the good case by elementary transformations. The argument is as follows.

Suppose first, for the sake of definiteness, that $i_2$ is not zero (evidently, no loss of generality results from such an assumption,
as the generators can be renamed).

As before, the action of $\Phi$ on
$$
\varphi: x_1\to x_1 + b x_2^{i_2}\ldots x_n^{i_n},\;\;x_2\to x_2,\;\;\ldots, \;\; x_n\to x_n
$$
yields an automorphism
$$
\Phi(\varphi): x_1 \to x_1 + ab x_2^{i_2}\ldots x_n^{i_n},\;\;x_2\to x_2,\;\;\ldots, \;\; x_n\to x_n.
$$
Now, for arbitrary $\lambda\in K^{\times}$, take a triangular linear automorphism
$$
\Lambda:x_2\to x_2 + \lambda x_3
$$
and consider
$$
\Lambda\circ\varphi\circ\Lambda^{-1}: x_1\to x_1 + b(x_2 + \lambda x_3)^{i_2}x_3^{i_3}\ldots x_n^{i_n}.
$$
The resulting automorphism is elementary of the form
$$
x_1\to x_1 + P(x_2,\ldots, x_n)
$$
with $P$ homogeneous of degree $(i_2+\cdots+i_n)$.

The key property, which is now manifest, is that $P$ contains good monomials -- one such monomial is
$$
\lambda b i_2x_2^{i_2-1}x_3^{i_3+1}x_4^{i_4}\ldots x_n^{i_n}\equiv -\lambda b x_2^{i_2-1}x_3^{i_3+1}x_4^{i_4}\ldots x_n^{i_n}
$$
as $i_2$ is congruent to $(-1)$ modulo $p\neq 2$.

The action of $\Phi$ leads to
$$
\Phi(\Lambda\circ\varphi\circ \Lambda^{-1})  = \Lambda\circ\Phi(\varphi)\circ \Lambda^{-1}
$$
and, as the automorphism acted upon is elementary, corresponds to multiplying every monomial of $P$ (left-hand side) by an
appropriate constant (specific to that monomial); by the preceding results every good monomial gets multiplied by one. As
before, the equality with the automorphism on the right-hand side is the equality between two polynomials. What is left is to
compare the coefficients of the monomial $x_2^{i_2-1}x_3^{i_3+1}x_4^{i_4}\ldots x_n^{i_n}$ on the two sides of the
equation. The left-hand side yields
$$
b\lambda
$$
(as the monomial is good), while the right-hand side, being the result of conjugation by $\Lambda$ of a single monomial $aM$
($a$ is the action of $\Phi$) yields
$$
ab\lambda
$$
from which we immediately obtain $a=1$. Thus the bad case is also processed.

Finally, combining the results of the previous subsection culminating in Theorem \ref{thmstabquadrmain} with the results we
have obtained leads to the proof of Theorem \ref{thmripsmain}. The main result of this paper is thus established.

\section{Discussion}

An immediate corollary of the main Theorem \ref{thmripsmain} is the following observation.

\begin{corollary} \label{corinnermain}
  Suppose $\Phi$ is an automorphism of $\TAut_0 K[x_1,\ldots, x_n]$ ($n>3$, $K$ algebraically closed of characteristic not two) which preserves every point of the maximal torus $T^n$. Then $\Phi$ is inner.
\end{corollary}

In the general case of $\Phi\in \Aut\Aut K[x_1,\ldots, x_n]$, the additional restriction of $\Phi$ being an $\Ind$-automorphism
is required for the topological techniques developed in \cite{KBYu} to work. At that point Theorem \ref{thmripsmain} becomes
(modulo conjugation by an automorphism) a statement on tame automorphism preservation.

The Rips conjecture (Conjecture \ref{conjrips}) remains a noteworthy combinatorial problem in positive characteristic. We defer
it, along with the resolution of the complications with the subspaces of homogeneous polynomials in positive characteristic, to
future work.

One final point we wish to make concerns our handling of the linear case (Theorem \ref{thmlinearstab}). We feel that
employment of heavy machinery (Theorem \ref{thmmcdwat} of McDonald and Waterhouse -- or, alternatively, results in
\cite{HaOM}) amounts to cracking a moderately hard nut with a comparatively alien (K-theoretic) sledgehammer. In all
likelihood, a purely combinatorial approach to tackle the linear case more in line with our standard techniques can be devised,
although it may possibly involve utilizing more elaborate torus actions or spaces $S$ of automorphisms defined by such (perhaps
not commuting with a fixed action, but rather obeying a more complicated identity involving it). We leave it as an exercise to the
interested reader.


\begin{thebibliography}{99}


\bibitem{An} D. Anick. Limits of tame automorphisms of $K[x_1,\ldots,x_N]$. J. Algebra,
    82, 459-468 (1983).


\bibitem{BCW} H. Bass, E.H. Connell, D. Wright. The Jacobian conjecture: reduction of degree and
    formal expansion of the inverse. Amer. Math. Soc. Bull., New Series, 7 (2) (1982): 287–330,
    doi:10.1090/S0273-0979-1982-15032-7, ISSN 1088-9485, MR 0663785.


\bibitem{BBRY} A. Belov, L. Bokut, L. Rowen and J.-T. Yu. The Jacobian Conjecture, together with
    Specht and Burnside-type problems. In Automorphisms in Birational and Affine Geometry (pp.
    249-285). Springer (2014).


\bibitem{TA1} A. Belov-Kanel, A. Elishev, F. Razavinia, Yu Jie-Tai and W. Zhang. Noncommutative
    Bia\l{}ynicki-Birula Theorem, Chebyshevskii Sbornik, 2020, 21(1): 51-61, arXiv: 1808.04903.


\bibitem{K-BK1} A. Belov-Kanel and M. Kontsevich. Automorphisms of the Weyl algebra. Lett. Math. Phys.
    74 (2005), 181-199.

\bibitem{K-BK2} A. Belov-Kanel and M. Kontsevich. The Jacobian Conjecture is stably equivalent to
    the Dixmier Conjecture. Mosc. Math. J., 7:2 (2007), 209--218; arXiv: math/0512171v2, 2005.


\bibitem{KBYu} A. Belov-Kanel, J.-T. Yu, A. Elishev. On the Augmentation Topology on Automorphism Groups of
    Affine Spaces and Algebras, arXiv:1207.2045, Int. J. Alg. Comp., https://doi.org/10.1142/S0218196718400040.


\bibitem{Bialynicki-Birula1} A. Bia\l{}ynicki-Birula. Remarks on the
 action of an algebraic torus on $k^{n}$, Bull. Acad.
Polon. Sci. Ser. Sci. Math. Astro. Phys. 14 (1966) 177-181.


\bibitem{Bialynicki-Birula2} A. Bia\l{}ynicki-Birula. Remarks on the
 action of an algebraic torus on $k^{n}$ II, Bull. Acad.
Polon. Sci. Ser. Sci. Math. Astr. Phys. 15 (1967) 123-125.


\bibitem{Dix} J. Dixmier. Sur les algebres de Weyl. Bull. Soc. Math. France, 96 (1968), 209-242.


\bibitem{vdE} A. van den Essen. Polynomial automorphisms and the Jacobian conjecture. Progress in
    Mathematics, 190. Birkhauser Verlag, Basel, 2000.


\bibitem{FurKr} J.-P. Furter and H. Kraft. On the geometry of the automorphism groups of affine
    varieties, arXiv:1809.04175.


\bibitem{HaOM} A. J. Hahn and O. T. O'Meara. The classical groups and K-theory, Springer-Verlag Berlin
    Heidelberg (1989).


\bibitem{KGE}   A. Kanel-Belov, S. Grigoriev, A. Elishev, J.-T. Yu and W. Zhang. Lifting of
    Polynomial Symplectomorphisms and Deformation Quantization, Comm. in Algebra,
    46:9 (2018), 3926-3938, arXiv:1707.06450.

\bibitem{Keller} O.-H. Keller. Ganze Cremona-Transformationen. Monatshefte Math. Phys., 47 (1):
    299–306, 1939.


\bibitem{McD} B. R. McDonald. Automorphisms of $\GL_n(R)$, Trans. Amer. Math. Soc., Vol 246 (1978), 155 --
    171.


\bibitem{Moh1} T.-T. Moh. On the Jacobian conjecture and the configurations of roots. J. reine
    angew. Math., 340 (340) (1983): 140–212, doi:10.1515/crll.1983.340.140, ISSN 0075-4102, MR
    0691964.


\bibitem{BIP2} B. I. Plotkin. Algebras with the same (algebraic) geometry. Proc. Steklov Inst.
    Math, Vol. 242, 2003, 165-196.

\bibitem{BIP1} B. I. Plotkin. Varieties of algebras and algebraic varieties. Israel J. of
    Mathematics, 1996, Vol. 96, No. 2, 511-522.

\bibitem{Shafarevich} I. R. Shafarevich.  On some infinite-dimensional groups II. Izv. Akad. Nauk
    SSSR Ser. Mat., 45:1 (1981), 214–226; Math. USSR-Izv., 18:1 (1982), 185–194.

\bibitem{Tsu2} Y. Tsuchimoto. Endomorphisms of Weyl algebra and $p$-curvatures. Osaka Journal of
    Mathematics, vol. 42, no. 2 (2005).

\bibitem{Tsu1} Y. Tsuchimoto. Preliminaries on Dixmier conjecture. Mem. Fac. Sci, Kochi Univ. Ser.
    A Math. 24 (2003), 43-59.


\bibitem{Wat} W. C. Waterhouse. Automorphisms of $\GL_n(R)$, Proc. Amer. Math. Soc., Vol. 79 No. 3 (Jul. 1980),
    347 -- 351.


\end{thebibliography}
\end{document}